\documentclass[11pt, a4paper, leqno]{amsart}

\usepackage[utf8]{inputenc}
\usepackage[T1]{fontenc}
\usepackage[english]{babel}
\usepackage{csquotes}

\usepackage[margin=1in]{geometry}
\usepackage{microtype}         
\usepackage{xcolor}            

\definecolor{color1}{HTML}{4169e1}
\definecolor{color2}{HTML}{B8B8B8}

\usepackage{enumitem}          
\setlist[enumerate,1]{label=(\arabic*)} 

\usepackage{graphicx}          
\usepackage{booktabs}          

\usepackage{etoolbox}          
\usepackage{chngcntr}          

\usepackage{mathtools}         
\usepackage{amssymb}           
\usepackage{bm}                
\usepackage{bbm}
\usepackage{mleftright}        
\mleftright

\usepackage{mathrsfs}          
\usepackage{upgreek}           

\usepackage{tikz}              
\usepackage{quiver}            

\newtheorem{theorem}{Theorem}[section]
\newtheorem{lemma}[theorem]{Lemma}
\newtheorem{proposition}[theorem]{Proposition}
\newtheorem{corollary}[theorem]{Corollary}

\theoremstyle{plain}
\newtheorem*{theoremA}{Theorem A}
\newtheorem*{corollaryB}{Corollary B}
\newtheorem*{theoremC}{Theorem C}
\newtheorem*{corollaryD}{Corollary D}
\newtheorem*{theoremE}{Theorem E}

\theoremstyle{definition}
\newtheorem{definition}[theorem]{Definition}

\newtheorem{remark}[theorem]{Remark} 

\theoremstyle{remark}

\counterwithout{equation}{section}

\mathtoolsset{showonlyrefs,showmanualtags}

\usepackage[colorlinks=true,
            linkcolor=color1,
            citecolor=color1,
            urlcolor=color2,
            pagebackref=true,
            unicode=true]{hyperref}

\usepackage[capitalise, nameinlink]{cleveref}

\crefname{theorem}{Theorem}{Theorems}
\crefname{lemma}{Lemma}{Lemmas}
\crefname{proposition}{Proposition}{Propositions}
\crefname{corollary}{Corollary}{Corollaries}
\crefname{conjecture}{Conjecture}{Conjectures}
\crefname{definition}{Definition}{Definitions}
\crefname{example}{Example}{Examples}
\crefname{problem}{Problem}{Problems}
\crefname{exercise}{Exercise}{Exercises}
\crefname{question}{Question}{Questions}
\crefname{remark}{Remark}{Remarks}
\crefname{notation}{Notation}{Notations}
\crefname{convention}{Convention}{Conventions}

\crefformat{equation}{(#2#1#3)}
\crefmultiformat{equation}{(#2#1#3)}{ and (#2#1#3)}{, (#2#1#3)}{, and (#2#1#3)}
\crefrangeformat{equation}{(#3#1#4)--(#5#2#6)}

\makeatletter
\newcommand\mkcal@one[1]{\expandafter\gdef\csname c#1\endcsname{\ensuremath{\mathcal{#1}}}}
\forcsvlist{\mkcal@one}{A,B,C,D,E,F,G,H,I,J,K,L,M,N,O,P,Q,R,S,T,U,V,W,X,Y,Z}

\newcommand\mksf@one[1]{\expandafter\gdef\csname s#1\endcsname{\ensuremath{\mathsf{#1}}}}
\forcsvlist{\mksf@one}{A,B,C,D,E,F,G,H,I,J,K,L,M,N,O,P,Q,R,S,T,U,V,W,X,Y,Z}

\newcommand\mkbb@one[1]{\expandafter\gdef\csname #1#1\endcsname{\ensuremath{\mathbb{#1}}}}
\forcsvlist{\mkbb@one}{A,B,C,D,E,F,G,H,I,J,K,L,M,N,O,P,Q,R,S,T,U,V,W,X,Y,Z}

\newcommand\mkbf@one[1]{\expandafter\gdef\csname b#1\endcsname{\ensuremath{\mathbf{#1}}}}
\forcsvlist{\mkbf@one}{A,B,C,D,E,F,G,H,I,J,K,L,M,N,O,P,Q,R,S,T,U,V,W,X,Y,Z}

\newcommand\mkrm@one[1]{\expandafter\gdef\csname r#1\endcsname{\ensuremath{\mathrm{#1}}}}
\forcsvlist{\mkrm@one}{A,B,C,D,E,F,G,H,I,J,K,L,M,N,O,P,Q,R,S,T,U,V,W,X,Y,Z}

\newcommand\mkfrak@one[1]{\expandafter\gdef\csname f#1\endcsname{\ensuremath{\mathfrak{#1}}}}
\forcsvlist{\mkfrak@one}{A,B,C,D,E,F,G,H,I,J,K,L,M,N,O,P,Q,R,S,T,U,V,W,X,Y,Z}
\makeatother

\DeclareMathOperator{\Db}{D^b}
\DeclareMathOperator{\Dperf}{D_{\mathrm{perf}}}

\DeclareMathOperator{\Stab}{Stab}

\newcommand{\K}{\rK}

\title[A remark on the full support property]{A remark on the full support property}
\author{Yiran Cheng}
\address{Department of Mathematics, Imperial College London, London SW7 2AZ, United Kingdom}
\email{y.cheng@imperial.ac.uk}

\begin{document}

\begin{abstract}
We show that a mass-Hom bound for a numerical pre-stability condition on a projective scheme over a field implies the support property with respect to the full numerical Grothendieck group.
Combined with the recent construction of stability conditions on projective schemes, this yields stability conditions with full support property on every projective scheme over a field.
We also establish a uniform version in the relative setting.
Finally, we show that the distinguished component in the full numerical stability manifold is independent of the choice of polarization, thereby defining a canonical component.
\end{abstract}

\maketitle

\section{Introduction}

In \cite[Theorem~1.1 and Remark~1.2]{LLL+}, Li--Liu--Liu--Macrì--Perry--Stellari--Zhao construct stability conditions on projective schemes over a field and, more generally, on projective families, and define distinguished connected components containing them.
They further prove that every stability condition in such a distinguished component admits a mass-Hom bound \cite[Theorem~1.3]{LLL+}.

Our first result shows that a mass-Hom bound already implies the support property with respect to the full numerical Grothendieck group.

\begin{theoremA}
[\textnormal{= Theorem~\ref{thm:mass-hom-implies-support}}]
Let $X$ be a projective scheme over a field $\mathbbm{k}$, and let $\sigma=(Z,\mathcal P)$ be a numerical pre-stability condition on $\Db(X)$.
If $\sigma$ admits a mass-Hom bound, then it satisfies the support property with respect to the full numerical Grothendieck group $\K_{\mathrm{num}}(X)$.
\end{theoremA}

\noindent
For an ample numerical class $H$ on $X$, write $\Stab_H^\dagger(\Db(X))$ for the distinguished component defined in \cite[Theorem~1.1 and Remark~1.2]{LLL+}.
Combining this result with \cite[Theorem~1.3]{LLL+} gives the following corollary, which answers the question about the full support property raised in \cite[Section~1.7]{LLL+}.

\begin{corollaryB}
[\textnormal{= Corollary~\ref{cor:absolute-full-support}}]
Every stability condition in $\Stab_H^\dagger(\Db(X))$ satisfies the support property with respect to $\K_{\mathrm{num}}(X)$.
In particular, every projective scheme over a field admits stability conditions satisfying the support property with respect to its numerical Grothendieck group.
\end{corollaryB}

The same argument admits a uniform relative version.
We formulate it using the uniformly numerical relative Grothendieck group $\mathcal N(X/S)$ of \cite[Proposition and Definition~21.5]{BLM+21} and introduce the notion of a \emph{uniform mass-Hom bound} in Definition~\ref{def:uniform-mass-hom}.

\begin{theoremC}
[\textnormal{= Theorem~\ref{thm:relative-mass-hom-support}}]
Let $\pi\colon X\to S$ be as in \cite[Setup~8.1]{LLL+}, and let $\sigma$ be a stability condition on $\Db(X)$ over $S$ whose central charge factors through $\mathcal N(X/S)$.
Assume that $\mathcal N(X/S)$ is a finite-rank free abelian group.
If $\sigma$ admits a uniform mass-Hom bound, then it satisfies the support property with respect to $\mathcal N(X/S)$ uniformly over $S$.
\end{theoremC}

\noindent
We moreover show that the distinguished relative component $\Stab^{\dagger}_{H_{X/S}}(\Db(X)/S)$ of \cite[Theorem~1.1]{LLL+} admits a uniform mass-Hom bound.
This gives the following.

\begin{corollaryD}
[\textnormal{= Corollary~\ref{cor:relative-full-support}}]
Let $\pi\colon X\to S$ be as in \cite[Setup~8.1]{LLL+}, and assume that $S$ is quasi-projective over a field.
Then every stability condition in the distinguished relative component $\Stab^{\dagger}_{H_{X/S}}(\Db(X)/S)$ satisfies the support property with respect to the uniformly numerical relative Grothendieck group $\mathcal N(X/S)$.
\end{corollaryD}

Finally, the full numerical support property allows us to compare the distinguished components associated with different polarizations.
Let $X$ be a projective scheme over a field, and write $\Stab_{\mathrm{num}}(\Db(X))$ for the space of stability conditions on $\Db(X)$ which satisfy the support property with respect to $\K_{\mathrm{num}}(X)$.
For $H\in\operatorname{Amp}(X)_{\mathbb Q}$, let $\Stab_{\mathrm{num},H}^\dagger(\Db(X))$ denote the connected component of $\Stab_{\mathrm{num}}(\Db(X))$ containing $\Stab_H^\dagger(\Db(X))$.
Our final result shows that this component is independent of $H$.

\begin{theoremE}
[\textnormal{= Theorem~\ref{thm:polarization-independent}}]
For any $H_0,H_1\in\operatorname{Amp}(X)_{\mathbb Q}$, we have
\begin{equation}
\Stab_{\mathrm{num},H_0}^\dagger(\Db(X))
=
\Stab_{\mathrm{num},H_1}^\dagger(\Db(X)).
\end{equation}
\end{theoremE}
\noindent
We therefore define $\Stab_{\mathrm{num}}^\dagger(\Db(X))$ as the connected component containing all $\Stab_H^\dagger(\Db(X))$.

\bigskip
\noindent
\textbf{Acknowledgements.}
I am grateful to Emanuele Macrì for his very helpful comments and suggestions. The author was supported by the Royal Society through the University Research Fellowship URF/R1/231191.


\section{Absolute case}\label{sec:absolute}

Throughout this section, let $X$ be a projective scheme over a field $\mathbbm{k}$.
We write $\Db(X)$ for the bounded derived category of coherent sheaves on $X$ and $\Dperf(X)$ for the category of perfect complexes on $X$.
Following \cite[Definition~2.7]{HLR}, we say that a numerical pre-stability condition $\sigma=(Z,\mathcal P)$ admits a \emph{mass-Hom bound} if, for every $A\in\Dperf(X)$, there exists a constant $C_A>0$ such that
\begin{equation}\label{eq:mass-hom}
\dim_{\mathbbm{k}}\operatorname{Hom}(A,F)\leq C_A\,m_\sigma(F)
\end{equation}
for every $F\in\Db(X)$, where $m_\sigma(F)$ denotes the mass of $F$ with respect to $\sigma$.

\begin{theorem}\label{thm:mass-hom-implies-support}
Let $X$ be a projective scheme over a field $\mathbbm{k}$, and let $\sigma=(Z,\mathcal P)$ be a numerical pre-stability condition on $\Db(X)$.
Suppose that $\sigma$ admits a mass-Hom bound.
Then $\sigma$ satisfies the support property with respect to the full numerical Grothendieck group $\K_{\mathrm{num}}(X)$.
\end{theorem}

Here, we set $\K_0(X)\coloneqq\K_0(\Db(X))$ and denote
\begin{equation}
\K_{\mathrm{num}}(X)=\K_0(X)/\ker\chi,
\end{equation}
where
\begin{equation}
\chi\colon\K_0(\Dperf(X))\times\K_0(X)\longrightarrow\mathbb Z,\qquad
\chi(A,F)=\sum_{j\in\mathbb Z}(-1)^j
\dim_{\mathbbm{k}}\operatorname{Hom}(A,F[j])
\end{equation}
is the Euler pairing and $\ker\chi$ denotes its right kernel.
By \cite[Lemma~12.7]{BLM+21}, $\K_{\mathrm{num}}(X)$ is a finite-rank free abelian group.
Throughout this section, we regard the central charge as a homomorphism
\begin{equation}
Z\colon\K_{\mathrm{num}}(X)\longrightarrow\mathbb C.
\end{equation}
In the notation of \cite{LLL+}, this is the composition of the central charge on some chosen lattice $\Lambda$ with the map $\K_{\mathrm{num}}(X)\xrightarrow{v}\Lambda$.

\begin{proof}[Proof of Theorem~\ref{thm:mass-hom-implies-support}]
By definition of $\K_{\mathrm{num}}(X)$, the functionals $u\mapsto\chi(A,u)$, for $A\in\Dperf(X)$, separate points of $\K_{\mathrm{num}}(X)_{\mathbb R}$.
We may therefore choose $A_1,\ldots,A_N\in\Dperf(X)$ such that
\begin{equation}
u\longmapsto\bigl(\chi(A_1,u),\ldots,\chi(A_N,u)\bigr)
\end{equation}
is injective.
Hence
\begin{equation}
\|u\|\coloneqq \left(\sum_{i=1}^N\chi(A_i,u)^2\right)^{1/2}
\end{equation}
defines a norm on $\K_{\mathrm{num}}(X)_{\mathbb R}$.

We claim that there exists $C>0$ such that
\begin{equation}\label{eq:support-norm}
\|[F]\|\leq C|Z(F)|
\end{equation}
for every $\sigma$-semistable object $F$.
Since both $\|[F]\|$ and $|Z(F)|$ are invariant under shifts, we may assume that $F\in\mathcal P(\phi)$ for some $0<\phi\leq1$.

For any fixed $A_i\in\{A_1,\ldots,A_N\}$,
if $\operatorname{Hom}(A_i,F[j])\neq0$, then the nonvanishing implies
\begin{equation}\label{eq:absolute-left}
j\geq\phi^-(A_i)-1.
\end{equation}
On the other hand,
let $\omega_X^\bullet$ be the dualizing complex on $X$.
Since $A_i$ is perfect and $\omega_X^\bullet$ has bounded coherent cohomology, we have $A_i\otimes\omega_X^\bullet\in\Db(X)$.
Grothendieck duality for $X\to\operatorname{Spec}(\mathbbm{k})$ (the usual Serre duality when $X$ is smooth) gives
\begin{equation}
\operatorname{Hom}(A_i,F[j])^\vee
\simeq
\operatorname{Hom}\bigl(F,A_i\otimes\omega_X^\bullet[-j]\bigr);
\end{equation}
see \cite[Tag~0AU3]{StacksProject}.
Thus nonvanishing also implies
\begin{equation}\label{eq:absolute-right}
j<\phi^+\bigl(A_i\otimes\omega_X^\bullet\bigr).
\end{equation}
It follows from \eqref{eq:absolute-left} and \eqref{eq:absolute-right} that only the integers in the finite set
\begin{equation}
J_{A_i}\coloneqq 
\left[\phi^-(A_i)-1,\,
\phi^+\bigl(A_i\otimes\omega_X^\bullet\bigr)\right)
\cap\mathbb Z
\end{equation}
can contribute to $\chi(A_i,F)$.
Therefore, by the mass-Hom bound, we have
\begin{equation}
|\chi(A_i,F)|
\leq
\sum_{j\in J_{A_i}}\dim_{\mathbbm{k}}\operatorname{Hom}(A_i[-j],F)
\leq
\widetilde C_{A_i}\,|Z(F)|,
\qquad \text{where }
\widetilde C_{A_i}\coloneqq \sum_{j\in J_{A_i}}C_{A_i[-j]},
\end{equation}
hence
\begin{equation}
\|[F]\|^2=\sum_{i=1}^N\chi(A_i,F)^2\leq C^2|Z(F)|^2,
\qquad \text{where }
C^2\coloneqq \sum_{i=1}^N\widetilde C_{A_i}^{\,2},
\end{equation}
which proves \eqref{eq:support-norm}.

Equivalently, the quadratic form
\begin{equation}
Q(u)\coloneqq C^2|Z(u)|^2-\|u\|^2
\end{equation}
is nonnegative on classes of $\sigma$-semistable objects and negative definite on $\ker Z$.
Thus $\sigma$ satisfies the support property; cf.~\cite[Lemma~A.4]{BMS16}.
\end{proof}

\begin{remark}
The statement of Theorem~\ref{thm:mass-hom-implies-support} in fact holds for any proper scheme $X$ over a field.
Indeed, the proof only requires that $\K_{\mathrm{num}}(X)$ be a finite-rank free abelian group, and this follows in the proper case from Chow's lemma together with the d\'evissage argument in the proof of \cite[Lemma~12.7]{BLM+21}.
\end{remark}

Following \cite{LLL+}, for any ample numerical class $H$ on $X$, we let $\operatorname{Stab}^{\dagger}_{H}(\Db(X))$ denote the distinguished component containing the stability conditions induced from those near the large volume limit on $\PP^n$ via a projective embedding $X\hookrightarrow\PP^n$ with hyperplane class $mH$ for some $m>0$.
By \cite[Theorem~7.5]{LLL+}, every stability condition in this component admits a mass-Hom bound.
Combining this with Theorem~\ref{thm:mass-hom-implies-support}, we obtain the following corollary.

\begin{corollary}\label{cor:absolute-full-support}
Every stability condition in $\operatorname{Stab}^{\dagger}_{H}(\Db(X))$ satisfies the support property with respect to the full numerical Grothendieck group $\K_{\mathrm{num}}(X)$.
\end{corollary}

\section{Relative case}\label{sec:relative}

Throughout this section, let $\pi\colon X\to S$ be a flat projective morphism, where $S$ is a noetherian Nagata scheme of finite Krull dimension which is quasi-projective over a noetherian affine scheme, as in \cite[Setup~8.1]{LLL+}.
We use the relative numerical Grothendieck group $\K_{\mathrm{num}}(X/S)$ of \cite[Definition~8.7]{LLL+}.
For $A\in\Dperf(X)$ and $s\in S$, we write $A_s$ for its restriction to $X_s$.
Following \cite[Proposition and Definition~21.5]{BLM+21}, we consider the relative Euler pairing
\begin{equation}
\chi\colon\K_0(\Dperf(X))\times\K_{\mathrm{num}}(X/S)
\longrightarrow\mathbb Z,\qquad
\chi(A,F)\coloneqq \chi(A_s,F)
\end{equation}
for $F\in\Db(X_s)$, and its uniformly numerical quotient
\begin{equation}
\mathcal N(X/S)=\K_{\mathrm{num}}(X/S)/\ker\chi,
\end{equation}
where $\ker\chi$ denotes the right kernel.
By \cite[Proposition and Definition~21.5]{BLM+21}, $\mathcal N(X/S)$ is a finite-rank free abelian group whenever $S$ is quasi-projective over a field.

Whenever the central charge of a relative stability condition factors through $\mathcal N(X/S)$, by abuse of notation we regard it as a homomorphism
\begin{equation}
Z\colon\mathcal N(X/S)\longrightarrow\mathbb C.
\end{equation}

\begin{definition}\label{def:uniform-mass-hom}
Let $\sigma=(\sigma_s)_{s\in S}$ be a stability condition on $\Db(X)$ over $S$ in the sense of \cite[Definition~8.17]{LLL+}.
We say that $\sigma$ admits a \emph{uniform mass-Hom bound} if, for every $A\in\Dperf(X)$, there exists $C_A>0$ such that
\begin{equation}\label{eq:uniform-mass-hom}
\dim_{\kappa(s)}\operatorname{Hom}(A_s,F)
\leq C_A\,m_{\sigma_s}(F)
\end{equation}
for every $s\in S$ and every $F\in\Db(X_s)$.
\end{definition}

\begin{theorem}\label{thm:relative-mass-hom-support}
Let $\sigma$ be a stability condition on $\Db(X)$ over $S$ whose central charge factors through $\mathcal N(X/S)$.
Assume that $\mathcal N(X/S)$ is a finite-rank free abelian group.
If $\sigma$ admits a uniform mass-Hom bound, then it satisfies the support property with respect to $\mathcal N(X/S)$ uniformly over $S$.
\end{theorem}

\begin{proof}
By definition of $\mathcal N(X/S)$, the functionals $u\mapsto\chi(A,u)$, for $A\in\Dperf(X)$, separate points of $\mathcal N(X/S)_{\mathbb R}$.
We may therefore choose $A_1,\ldots,A_N\in\Dperf(X)$ such that
\begin{equation}
u\longmapsto\bigl(\chi(A_1,u),\ldots,\chi(A_N,u)\bigr)
\end{equation}
is injective.
Hence
\begin{equation}
\|u\|\coloneqq \left(\sum_{i=1}^N\chi(A_i,u)^2\right)^{1/2}
\end{equation}
defines a norm on $\mathcal N(X/S)_{\mathbb R}$.

We claim that there exists $C>0$, independent of $s$, such that
\begin{equation}\label{eq:relative-support-norm}
\|[F]\|\leq C|Z(F)|
\end{equation}
for every $s\in S$ and every $\sigma_s$-semistable object $F$.
Since both $\|[F]\|$ and $|Z(F)|$ are invariant under shifts, we may assume that $F\in\Db(X_s)$ is $\sigma_s$-semistable of phase in $(0,1]$.

Let $\omega^\bullet_{X/S}$ be the relative dualizing complex, which is $S$-perfect by \cite[Tag~0E2Z]{StacksProject}.
For any fixed $A_i\in\{A_1,\ldots,A_N\}$, both $A_i$ and $A_i\otimes\omega^\bullet_{X/S}$ are $S$-perfect.
By \cite[Lemma~8.21]{LLL+}, the corresponding phase functions are constructible and semicontinuous.
Since $S$ is quasi-compact, there exist $a_{A_i},b_{A_i}\in\mathbb R$, independent of $s$, such that
\begin{equation}
a_{A_i}\leq\phi^-_{\sigma_s}((A_i)_s),\qquad
\phi^+_{\sigma_s}
\bigl((A_i\otimes\omega^\bullet_{X/S})_s\bigr)\leq b_{A_i}
\end{equation}
for every $s\in S$.
If $\operatorname{Hom}((A_i)_s,F[j])\neq0$, then
\begin{equation}\label{eq:relative-left}
j\geq a_{A_i}-1.
\end{equation}
On the other hand, by base change for the relative dualizing complex \cite[Tag~0E4P]{StacksProject} and fiberwise Grothendieck duality \cite[Tag~0AU3]{StacksProject},
\begin{equation}
\operatorname{Hom}((A_i)_s,F[j])^\vee
\simeq
\operatorname{Hom}
\bigl(F,(A_i\otimes\omega^\bullet_{X/S})_s[-j]\bigr).
\end{equation}
Thus nonvanishing also implies
\begin{equation}\label{eq:relative-right}
j<b_{A_i}.
\end{equation}
It follows from \eqref{eq:relative-left} and \eqref{eq:relative-right} that only the integers in the finite set
\begin{equation}
J_{A_i}\coloneqq 
[a_{A_i}-1,b_{A_i})\cap\mathbb Z
\end{equation}
can contribute to $\chi(A_i,F)$, independently of $s$.
Therefore, by the uniform mass-Hom bound, we have
\begin{equation}
|\chi(A_i,F)|
\leq
\sum_{j\in J_{A_i}}
\dim_{\kappa(s)}\operatorname{Hom}((A_i)_s[-j],F)
\leq
\widetilde C_{A_i}\,|Z(F)|,
\qquad \text{where }
\widetilde C_{A_i}\coloneqq \sum_{j\in J_{A_i}}C_{A_i[-j]},
\end{equation}
hence
\begin{equation}
\|[F]\|^2
=
\sum_{i=1}^N\chi(A_i,F)^2
\leq C^2|Z(F)|^2,
\qquad \text{where }
C^2\coloneqq \sum_{i=1}^N\widetilde C_{A_i}^{\,2},
\end{equation}
which proves \eqref{eq:relative-support-norm} uniformly over $S$.

Equivalently, the quadratic form
\begin{equation}
Q(u)\coloneqq C^2|Z(u)|^2-\|u\|^2
\end{equation}
is nonnegative on all fiberwise semistable classes and negative definite on $\ker Z$.
Thus $\sigma$ satisfies the support property with respect to $\mathcal N(X/S)$ uniformly over $S$; cf.~\cite[Lemma~A.4]{BMS16}.
\end{proof}

\begin{remark}\label{rem:proper-relative}
The argument of Theorem~\ref{thm:relative-mass-hom-support} in fact applies to flat proper families, provided that the relevant relative numerical lattice is a finite-rank free abelian group and that the phases of $S$-perfect objects admit uniform bounds over the base.
\end{remark}

Following \cite{LLL+}, for any relatively ample numerical class $H_{X/S}$, we let $\operatorname{Stab}^{\dagger}_{H_{X/S}}(\Db(X)/S)$ denote the distinguished component of \cite[Section~10.2]{LLL+}, containing the explicit stability conditions $\widetilde\sigma_{a,b}^{X/S}$ for $a\gg0$.
We now establish a uniform relative version of the mass-Hom bound in \cite[Theorem~7.5]{LLL+}.

\begin{proposition}\label{prop:uniform-mass-hom}
Every stability condition in $\operatorname{Stab}^{\dagger}_{H_{X/S}}(\Db(X)/S)$ admits a uniform mass-Hom bound.
\end{proposition}

\begin{proof}
We show that the argument of \cite[Theorem~7.5]{LLL+} can be made uniform over the base.

First, uniform mass-Hom bounds are preserved by the pullback and pushforward constructions used in \cite[Proposition~10.2 and Theorem~10.3]{LLL+}.
Indeed, the adjunction estimates of \cite[Lemma~7.4]{LLL+} apply fiberwise to fixed perfect generators on the total spaces, with constants independent of $s$.
Moreover, any fixed finite construction from these generators by shifts, cones, and direct summands restricts to the same construction on every fiber.
Thus the argument of \cite[Remark~7.2]{LLL+} is uniform over $S$.

Second, uniform mass-Hom bounds are preserved along connected components of the relative stability space.
For sufficiently close relative stability conditions $\sigma$ and $\tau$, the deformation estimates in \cite[Theorem~22.2 and Lemma~22.3]{BLM+21}, together with the mass-comparison argument of \cite[Remark~7.3]{LLL+}, give a constant $C_{\sigma,\tau}>0$, independent of $s$, such that
\begin{equation}\label{eq:uniform-mass-comparison}
m_{\sigma_s}(F)\leq C_{\sigma,\tau}\,m_{\tau_s}(F)
\end{equation}
for every $s\in S$ and $F\in\Db(X_s)$.
Hence the uniform mass-Hom bound is preserved under small deformations; since the relative stability space is locally path connected by \cite[Theorem~8.19]{LLL+}, a finite covering of a path proves the claim for each connected component.

It remains to find one relative stability condition with a uniform mass-Hom bound on $(\PP^1_S)^n$.
Choose $(a,b_0)$ as in the proof of \cite[Proposition~10.1]{LLL+}.
There is $G\in\widetilde{\operatorname{GL}}_2^+(\mathbb R)$, independent of $s$, such that $G\cdot\sigma_{a,b_0}^{(\PP^1_s)^n}$ is algebraic for every $s$.
Let $Z_G$ denote the resulting central charge and, for $I\subset\{1,\ldots,n\}$, set
\[
L_I\coloneqq \mathcal O_{(\PP^1_S)^n}(I)[n-|I|].
\]
By \cite[Section~5.1]{LLL+}, the corresponding algebraic heart is the extension closure of the $(L_I)_s$.
The values $Z_G((L_I)_s)$ are independent of $s$; we write them as $Z_G(L_I)$.
Since $Z_G(L_I)\in\mathbb H\cup\mathbb R_{<0}$ for every $I$, the continuous function $(x_I)_I\longmapsto \left|\sum_I x_I Z_G(L_I)\right|$ has a positive minimum on the compact convex hull $x_I\geq0$, $\sum_Ix_I=1$.
Hence there is $c>0$ such that
\begin{equation}\label{eq:algebraic-length-mass}
\sum_I n_I
\leq
c\left|\sum_I n_I Z_G(L_I)\right|
\end{equation}
holds for every choice of $n_I\geq0$.

For each $I$ and $J$, the quantity
\[
\sum_{\ell\in\mathbb Z}
\dim_{\kappa(s)}
\operatorname{Hom}\bigl((L_I)_s,(L_J)_s[\ell]\bigr)
\]
is finite and independent of $s$ by the standard cohomology of line bundles on $(\PP^1_{\kappa(s)})^n$.
For each $I$, set
\[
c_I\coloneqq 
\max_J
\sum_{\ell\in\mathbb Z}
\dim_{\kappa(s)}
\operatorname{Hom}\bigl((L_I)_s,(L_J)_s[\ell]\bigr).
\]
If $F=F_0[q]$ is stable for $G\cdot\sigma_{a,b_0}^{(\PP^1_s)^n}$, with $F_0$ in the algebraic heart, choose a finite filtration of $F_0$ with factors among the $(L_J)_s$, and let $n_J$ be their multiplicities.
Then $Z_G(F_0)=\sum_J n_J Z_G(L_J)$ and
\[
m_{G\cdot\sigma_{a,b_0}^{(\PP^1_s)^n}}(F)=|Z_G(F_0)|.
\]
Subadditivity of total graded Hom dimension together with \eqref{eq:algebraic-length-mass} gives
\begin{equation}\label{eq:generator-uniform-bound}
\dim_{\kappa(s)}
\operatorname{Hom}\bigl((L_I)_s,F\bigr)
\leq
c_I\sum_J n_J
\leq
c\,c_I\,
m_{G\cdot\sigma_{a,b_0}^{(\PP^1_s)^n}}(F).
\end{equation}

Now fix $A\in\Dperf((\PP^1_S)^n)$ and write $\pi\colon(\PP^1_S)^n\to S$ for the projection.
The relative full exceptional collection shows that $A$ can be obtained, by a fixed finite construction using shifts, cones, and direct summands, from objects $\pi^*P_k\otimes L_{I_k}$ with $P_k\in\Dperf(S)$.
Since $P_k$ is perfect and $S$ is quasi-compact, the total dimension of $H^*((P_k)_s)$ is bounded independently of $s$; over $\kappa(s)$, $(P_k)_s$ splits as the direct sum of its cohomology groups in the corresponding degrees.
Hence \eqref{eq:generator-uniform-bound} and the fixed construction of $A$ give a constant $C_A>0$, independent of $s$, such that
\[
\dim_{\kappa(s)}\operatorname{Hom}(A_s,F)
\leq
C_A\,m_{G\cdot\sigma_{a,b_0}^{(\PP^1_s)^n}}(F)
\]
for every $s$ and every stable $F$.
Applying the argument of \cite[Lemma~2.8]{HLR} fiberwise with the same constants, the same holds for arbitrary $F$.
Thus $G\cdot\sigma_{a,b_0}^{(\PP^1_S)^n}$ admits a uniform mass-Hom bound.

Since $G$ is fixed, the masses before and after applying $G$ are uniformly comparable, so $\sigma_{a,b_0}^{(\PP^1_S)^n}$ also admits a uniform mass-Hom bound.
By \cite[Proposition~10.1]{LLL+} and the connected-component argument above, every $\sigma_{a,b}^{(\PP^1_S)^n}$ admits a uniform mass-Hom bound.
By pullback and pushforward, together with \cite[Proposition~10.2 and Theorem~10.3]{LLL+}, the stability conditions $\widetilde\sigma_{a,b}^{X/S}$ for $a\gg0$ admit a uniform mass-Hom bound.
The connected-component argument then proves the claim.
\end{proof}

\begin{corollary}\label{cor:relative-full-support}
Assume that $\mathcal N(X/S)$ is a finite-rank free abelian group.
Then every stability condition in $\operatorname{Stab}^{\dagger}_{H_{X/S}}(\Db(X)/S)$ satisfies the support property with respect to $\mathcal N(X/S)$ uniformly over $S$.
In particular, the conclusion holds whenever $S$ is quasi-projective over a field.
\end{corollary}

\begin{proof}
By Proposition~\ref{prop:uniform-mass-hom}, these stability conditions admit a uniform mass-Hom bound.
By construction in \cite[Section~10.2]{LLL+}, the homomorphism $v_{H_{X/S}}$ factors through $\mathcal N(X/S)$.
The result therefore follows from Theorem~\ref{thm:relative-mass-hom-support}.
The final assertion follows from \cite[Proposition and Definition~21.5]{BLM+21}.
\end{proof}

\section{The distinguished connected component} \label{sec:connected}

We write $\Stab_{\mathrm{num}}(\Db(X))$ for the space of stability conditions on $\Db(X)$ which satisfy the support property with respect to $\K_{\mathrm{num}}(X)$.
For any $H\in\operatorname{Amp}(X)_{\mathbb Q}$, the distinguished component $\Stab_H^\dagger(\Db(X))$ of \cite{LLL+} is contained in $\Stab_{\mathrm{num}}(\Db(X))$, by Corollary~\ref{cor:absolute-full-support}.
We denote by $\Stab_{\mathrm{num},H}^\dagger(\Db(X))$ the connected component of $\Stab_{\mathrm{num}}(\Db(X))$ containing $\Stab_H^\dagger(\Db(X))$.

\begin{definition}
We define
\begin{equation}
\Stab_{\mathrm{num}}^\dagger(\Db(X))\coloneqq \bigcup_{H\in\operatorname{Amp}(X)_{\mathbb Q}}\Stab_{\mathrm{num},H}^\dagger(\Db(X)).
\end{equation}
A priori, this is a union of connected components of $\Stab_{\mathrm{num}}(\Db(X))$.
\end{definition}

\begin{theorem}
\label{thm:polarization-independent}
The space $\Stab_{\mathrm{num}}^\dagger(\Db(X))$ is a single connected component of $\Stab_{\mathrm{num}}(\Db(X))$.
Equivalently, for any $H_0,H_1\in\operatorname{Amp}(X)_{\mathbb Q}$, we have
\begin{equation}
\Stab_{\mathrm{num},H_0}^\dagger(\Db(X))=\Stab_{\mathrm{num},H_1}^\dagger(\Db(X)).
\end{equation}
\end{theorem}

The idea is to embed $X$ into a fixed product $\mathbb P^r\times\mathbb P^s$ so that the segment joining two polarizations is induced by a path of polarizations on the ambient product.
Before proving the theorem, we establish two auxiliary results concerning large-volume stability conditions on products of projective spaces.
We use the following notation throughout.
On $\mathbb P^r\times\mathbb P^s$, set
\begin{equation}
\xi\coloneqq c_1(\mathcal O(1,0)), \qquad \eta\coloneqq c_1(\mathcal O(0,1)),
\end{equation}
where
$\mathcal O(p,q)\coloneqq \operatorname{pr}_1^*\mathcal O_{\mathbb P^r}(p)\otimes\operatorname{pr}_2^*\mathcal O_{\mathbb P^s}(q)$.
For any $u\in\K_{\mathrm{num}}(\mathbb P^r\times\mathbb P^s)_{\mathbb R}$, write
\begin{equation}
\operatorname{ch}(u)=\sum_{i=0}^r\sum_{j=0}^s u_{ij}\xi^i\eta^j.
\end{equation}
For $a,\lambda,\mu>0$, define a linear transformation
\begin{equation}\label{eq:block-scaling}
M_a^{(\lambda,\mu)}(u)\coloneqq \left((a\lambda)^{r-i}(a\mu)^{s-j}u_{ij}\right)_{\substack{0\leq i\leq r\\0\leq j\leq s}}\in\mathbb R^{(r+1)(s+1)}.
\end{equation}
After fixing a Euclidean norm $\|\cdot\|$ on $\mathbb R^{(r+1)(s+1)}$, set
\begin{equation}
\|u\|_a^{(\lambda,\mu)}\coloneqq \left\|M_a^{(\lambda,\mu)}(u)\right\|.
\end{equation}

We first state a product version of the construction of geometric stability conditions on projective spaces in \cite{LLL+}.

\begin{proposition}
\label{prop:product-projective-spaces}
For every $a,\lambda,\mu>0$, there is a stability condition
\begin{equation}
\sigma_a^{(\lambda,\mu)}
=
\bigl(Z_a^{(\lambda,\mu)},\mathcal P_a^{(\lambda,\mu)}\bigr)
\end{equation}
on $\Db(\PP^r\times\PP^s)$ with
\begin{equation}\label{eq:block-central-charge}
Z_a^{(\lambda,\mu)}(F)
=
-\int_{\PP^r\times\PP^s}
e^{-\sqrt{-1}a(\lambda\xi+\mu\eta)}
\operatorname{ch}(F),
\end{equation}
such that skyscraper sheaves have phase $1$.
These stability conditions depend continuously on $(a,\lambda,\mu)$ and satisfy the support property with respect to $\K_{\mathrm{num}}(\PP^r\times\PP^s)$.
Moreover, the following properties hold.
\begin{enumerate}
\item[(i)]
For every fixed $\lambda,\mu>0$, there is a constant $C_{\lambda,\mu}>0$ such that
\begin{equation}
\|u\|_a^{(\lambda,\mu)}
\leq
C_{\lambda,\mu}|Z_a^{(\lambda,\mu)}(u)|
\end{equation}
for every $a\geq1$ and every $\sigma_a^{(\lambda,\mu)}$-semistable class $u$.

\item[(ii)]
For every compact subset $K\subset\mathbb R_{>0}^2$ and every $p,q\in\mathbb Z$, we have
\begin{equation}
\sup_{(\lambda,\mu)\in K}
\operatorname{dist}
\bigl(
\mathcal P_a^{(\lambda,\mu)},
\mathcal P_a^{(\lambda,\mu)}\otimes\mathcal O(p,q)
\bigr)
\longrightarrow0
\end{equation}
as $a\to\infty$.
\end{enumerate}
\end{proposition}

\begin{proof}
Suppose first that $a\lambda,a\mu\in\mathbb Q_{>0}$.
The proof of \cite[Theorem~4.2]{LLL+}, together with the support-property argument of \cite[Theorem~4.5]{LMPSZ}, extends to factor-dependent positive rational weights.
More precisely, one argues inductively with a positive rational weight assigned to each elliptic-curve factor.
For every choice of an omitted factor, the inductive construction gives a stability condition with the same weighted central charge, while skyscraper sheaves remain stable of phase $1$.
Thus \cite[Theorem~4.1(2)]{LLL+} shows that the resulting stability conditions coincide, and the intersection-of-kernels argument in the proof of \cite[Theorem~4.5]{LMPSZ} gives the support property with respect to the full lattice used there.

Let $E$ be the elliptic curve used in the construction of \cite[Theorem~4.2]{LLL+}, and assign weight $2a\lambda$ to each of the first $r$ factors of $E^{r+s}$ and weight $2a\mu$ to each of the last $s$ factors.
Iterating the product construction gives the corresponding weighted stability condition on $\Db(E^{r+s})$, with skyscraper sheaves stable of phase $1$.
By the same invariance argument as in the proof of \cite[Theorem~5.6]{LLL+}, this stability condition is invariant under
\begin{equation}
\bigl((\mathbb Z/2\mathbb Z)^r\rtimes\mathfrak S_r\bigr)
\times
\bigl((\mathbb Z/2\mathbb Z)^s\rtimes\mathfrak S_s\bigr).
\end{equation}
By \cite[Proposition~5.1]{LLL+}, the resulting stability condition pushes forward along
\begin{equation}
E^{r+s}\longrightarrow(\PP^1)^{r+s}
\end{equation}
to a numerical stability condition whose slicing satisfies the Bayer property with respect to every effective line bundle.
We then apply the proof of \cite[Theorem~5.6]{LLL+} successively to the two blocks.
For the first quotient
\begin{equation}
(\PP^1)^{r+s}\longrightarrow\PP^r\times(\PP^1)^s,
\end{equation}
the filtration property of \cite[Example~3.20(3)]{LLL+} is preserved after taking the product with $(\PP^1)^s$, the required Bayer properties follow from \cite[Proposition~5.1]{LLL+}, and the slicing is $\mathfrak S_r$-invariant by the same argument as in the proof of \cite[Theorem~5.6]{LLL+}.
Thus \cite[Proposition~3.21]{LLL+} gives a stability condition on $\PP^r\times(\PP^1)^s$.
By \cite[Lemma~3.9(1) and Remark~3.14(3)]{LLL+}, the Bayer properties for effective line bundles in the second block are preserved under this pushforward, while the $\mathfrak S_s$-invariance is preserved since the two block actions commute.
Applying \cite[Proposition~3.21]{LLL+} once more to
\begin{equation}
\PP^r\times(\PP^1)^s\longrightarrow\PP^r\times\PP^s
\end{equation}
therefore gives the desired stability condition.
After normalizing by the degree of the composite quotient $E^{r+s}\to\PP^r\times\PP^s$, the central charge takes the form \eqref{eq:block-central-charge}.

The numerical lattice obtained after passing to the two block symmetries has one coordinate for each pair
\begin{equation}
(i,j),\qquad 0\leq i\leq r,\quad 0\leq j\leq s.
\end{equation}
Under the Chern character, these coordinates identify over $\mathbb R$ with the coordinates $u_{ij}$ introduced above.
They therefore separate classes in $\K_{\mathrm{num}}(\PP^r\times\PP^s)_{\mathbb R}$, so the resulting stability condition satisfies the support property with respect to $\K_{\mathrm{num}}(\PP^r\times\PP^s)$.

The extension to arbitrary $a,\lambda,\mu>0$, together with continuity of the resulting family, follows from the same deformation-and-rescaling argument as in \cite[Proposition~4.5]{LLL+}, using multiplication isogenies separately on the two blocks.

To prove (i), note that \eqref{eq:block-central-charge} can be written as
\begin{equation}
Z_a^{(\lambda,\mu)}(u)
=
-\sum_{i=0}^r\sum_{j=0}^s
\frac{(-\sqrt{-1})^{r+s-i-j}}{(r-i)!(s-j)!}
\bigl(M_a^{(\lambda,\mu)}(u)\bigr)_{ij}.
\end{equation}
Thus, in the scaled coordinates determined by $M_a^{(\lambda,\mu)}$, the central charge is independent of $a,\lambda,\mu$.
For fixed $\lambda,\mu$, the same induction as in the proof of \cite[Lemma~6.6]{LLL+}, with the corresponding blockwise scaling, gives a quadratic form in the scaled coordinates which is independent of $a$, nonnegative on the scaled classes of $\sigma_a^{(\lambda,\mu)}$-semistable objects, and negative definite on the kernel of this fixed linear functional.
The same homogeneity and compactness argument therefore gives a constant $C_{\lambda,\mu}>0$, independent of $a$, such that
\begin{equation}
\left\|M_a^{(\lambda,\mu)}(u)\right\|
\leq
C_{\lambda,\mu}|Z_a^{(\lambda,\mu)}(u)|
\end{equation}
for every $\sigma_a^{(\lambda,\mu)}$-semistable class $u$.
This proves (i).

Finally, to prove (ii), the same construction with independent real $B$-field parameters in the two blocks gives a continuous deformation family.
Applying the rescaling argument using multiplication isogenies in the proof of \cite[Theorem~6.8]{LLL+} separately to the two groups of elliptic-curve factors, tensoring by $\mathcal O(p,q)$ produces $B$-field parameters of order $(a\lambda)^{-1}$ and $(a\mu)^{-1}$ in the two blocks.
Since $K\subset\mathbb R_{>0}^2$ is compact, these parameters tend uniformly to zero on $K$, and for $a\gg0$ the corresponding points lie in a fixed compact subset of the deformation space.
Uniform continuity on this compact subset therefore gives (ii).
\end{proof}

\begin{remark}
We have stated Proposition~\ref{prop:product-projective-spaces}(ii) in this asymptotic form in order to keep the argument parallel to the proof of \cite[Theorem~6.8]{LLL+}.
In fact, the argument of \cite[Theorem~2.4]{CF26}, applied separately to the two blocks, gives the stronger quantitative estimate
\begin{equation}
\operatorname{dist}
\bigl(
\mathcal P_a^{(\lambda,\mu)},
\mathcal P_a^{(\lambda,\mu)}\otimes\mathcal O(p,q)
\bigr)
\leq
\frac{1}{\uppi a}
\left(
\frac{r|p|}{\lambda}
+
\frac{s|q|}{\mu}
\right).
\end{equation}
\end{remark}

For later use, we denote by $\sigma_a^{\mathbb P^N}$ the stability condition on $\Db(\mathbb P^N)$ constructed in \cite[Theorem~5.6]{LLL+} with central charge
\begin{equation}
Z_a^{\mathbb P^N}(F)\coloneqq -\int_{\mathbb P^N}e^{-\sqrt{-1}aH}\operatorname{ch}(F),
\end{equation}
where $H$ is the hyperplane class, and for which skyscraper sheaves have phase $1$.

\begin{lemma}[Segre--Veronese compatibility]
\label{lem:segre-veronese-compatibility}
Let $\lambda,\mu$ be positive integers, and let
\begin{equation}
\jmath\colon\mathbb P^r\times\mathbb P^s\hookrightarrow\mathbb P^N
\end{equation}
be the Segre--Veronese embedding defined by $\cO(\lambda,\mu)$.
For $a\gg0$, the stability conditions
\begin{equation}
\sigma_a^{(\lambda,\mu)} \qquad\text{and}\qquad \widetilde{\sigma}_a\coloneqq \jmath^\sharp\sigma_a^{\mathbb P^N}
\end{equation}
lie in the same connected component of $\Stab_{\mathrm{num}}(\Db(\mathbb P^r\times\mathbb P^s))$, where $\jmath^\sharp$ denotes the pullback construction for stability conditions of \cite[Section~3]{LLL+}.

Moreover, given finitely many pairs $(p_1,q_1),\ldots,(p_k,q_k)\in\mathbb Z^2$, for $a\gg0$ these two stability conditions can be joined by a path
\begin{equation}
\sigma_a^t=(Z_a^t,\mathcal P_a^t), \qquad t\in[0,1],
\end{equation}
such that
\begin{equation}\label{eq:compatibility-bayer}
\operatorname{dist}\bigl(\mathcal P_a^t,\mathcal P_a^t\otimes\mathcal O(p_j,q_j)\bigr)<1
\end{equation}
for every $t\in[0,1]$ and every $j=1,\ldots,k$.
\end{lemma}

\begin{proof}
The argument is a product generalization of the proof of \cite[Theorem~6.8]{LLL+}.
Since $\lambda$ and $\mu$ are fixed throughout the proof, write
\begin{equation}
\sigma_a\coloneqq \sigma_a^{(\lambda,\mu)}=(Z_a,\mathcal P_a), \qquad \|-\|_a\coloneqq \|-\|_a^{(\lambda,\mu)}.
\end{equation}
Thus
\begin{equation}
Z_a(F)=-\int_{\mathbb P^r\times\mathbb P^s}e^{-\sqrt{-1}a(\lambda\xi+\mu\eta)}\operatorname{ch}(F).
\end{equation}
Let $\widetilde Z_a$ be the central charge obtained by restricting $Z_a^{\mathbb P^N}$ along $\jmath$.
By Grothendieck--Riemann--Roch, we have
\begin{equation}
\widetilde Z_a(F)=-\int_{\mathbb P^r\times\mathbb P^s}e^{-\sqrt{-1}a(\lambda\xi+\mu\eta)}\operatorname{ch}(F)\operatorname{td}(N_\jmath)^{-1},
\end{equation}
where $N_\jmath$ is the normal bundle of $\jmath$.
Consider the linear interpolation
\begin{equation}\label{eq:central-charge-interpolation}
Z_a^t\coloneqq (1-t)Z_a+t\widetilde Z_a, \qquad t\in[0,1].
\end{equation}
Since $\operatorname{td}(N_\jmath)^{-1}-1$ is a finite linear combination of monomials $\xi^\alpha\eta^\beta$ with $\alpha+\beta>0$, it is enough to consider one such monomial.
Multiplication by $\xi^\alpha\eta^\beta$ sends the coefficient $u_{ij}$ to the $(i+\alpha,j+\beta)$-coordinate, whose block-scaled coefficient differs by the factor $(a\lambda)^{-\alpha}(a\mu)^{-\beta}$.
Since $\lambda,\mu$ are fixed, this is $O(a^{-1})$.
Hence there is a constant $C>0$, independent of $a$ and $t$, such that
\begin{equation}\label{eq:interpolation-error}
|Z_a^t(u)-Z_a(u)|\leq\frac{C}{a}\|u\|_a
\end{equation}
holds for every $u\in\K_{\mathrm{num}}(\mathbb P^r\times\mathbb P^s)_{\mathbb R}$, $t\in[0,1]$, and $a\geq1$.

By Proposition~\ref{prop:product-projective-spaces}(i), there is a constant $C_0>0$, independent of $a$, such that every $\sigma_a$-semistable class $u$ satisfies
\begin{equation}
\|u\|_a\leq C_0|Z_a(u)|.
\end{equation}
Set
\begin{equation}
Q_a(u)\coloneqq C_0^2|Z_a(u)|^2-\|u\|_a^2.
\end{equation}
Then $Q_a$ is a support form for $\sigma_a$.
If $Z_a^t(u)=0$, then \eqref{eq:interpolation-error} gives
\begin{equation}
Q_a(u)\leq\left(\frac{C_0^2C^2}{a^2}-1\right)\|u\|_a^2.
\end{equation}
Thus, for $a\gg0$, $Q_a$ is negative definite on $\ker Z_a^t$ for every $t\in[0,1]$.
By \cite[Proposition~A.5]{BMS16}, the path \eqref{eq:central-charge-interpolation} therefore lifts from $\sigma_a$ to a path
\begin{equation}
\sigma_a^t=(Z_a^t,\mathcal P_a^t), \qquad t\in[0,1],
\end{equation}
such that every $\sigma_a^t$ satisfies the support property with respect to $Q_a$.
If $Q_a(u)\geq0$, then $\|u\|_a\leq C_0|Z_a(u)|$, so \eqref{eq:interpolation-error} gives
\begin{equation}
|Z_a^t(u)-Z_a(u)|\leq\frac{CC_0}{a}|Z_a(u)|.
\end{equation}
Taking $a$ sufficiently large that $CC_0/a<\sin(\uppi/8)$, the same evaluation argument as in the proof of \cite[Theorem~6.8]{LLL+}, using \cite[Lemma~2.9]{Bayer:2019}, provides
\begin{equation}\label{eq:compatibility-close}
\operatorname{dist}\bigl(\mathcal P_a^t,\mathcal P_a\bigr)<\frac18
\end{equation}
for every $t\in[0,1]$.

By Proposition~\ref{prop:product-projective-spaces}(ii), after increasing $a$ if necessary, we have
\begin{equation}
\operatorname{dist}\bigl(\mathcal P_a,\mathcal P_a\otimes\mathcal O(-\lambda,-\mu)\bigr)<\frac12.
\end{equation}
Since tensoring by a line bundle preserves the distance between slicings, \eqref{eq:compatibility-close} gives
\begin{equation}
\operatorname{dist}\bigl(\mathcal P_a^t,\mathcal P_a^t\otimes\mathcal O(-\lambda,-\mu)\bigr)<\frac34<1.
\end{equation}
Hence \cite[Remark~3.18]{LLL+} gives
\begin{equation}\label{eq:compatibility-strict-bayer}
\mathcal P_a^t\prec\mathcal P_a^t\otimes\mathcal O(-\lambda,-\mu)[1].
\end{equation}

For $a\gg0$, \cite[Theorem~6.2]{LLL+} gives
\begin{equation}
\widetilde{\sigma}_a=(\widetilde Z_a,\widetilde{\mathcal P}_a)=\jmath^\sharp\sigma_a^{\mathbb P^N},
\end{equation}
with skyscraper sheaves of phase $1$ and
\begin{equation}
\widetilde{\mathcal P}_a\prec\widetilde{\mathcal P}_a\otimes\mathcal O(-\lambda,-\mu)[1].
\end{equation}
By \eqref{eq:compatibility-strict-bayer} and \cite[Proposition~3.27]{LLL+}, skyscraper sheaves are $\sigma_a^1$-semistable.
Since $Z_a^1=\widetilde Z_a$ and $\widetilde Z_a(\mathcal O_x)=-1$, while \eqref{eq:compatibility-close} places their $\sigma_a^1$-phase in $(7/8,9/8)$, they have phase $1$.
Thus $\sigma_a^1$ and $\widetilde{\sigma}_a$ satisfy the hypotheses of \cite[Lemma~6.7]{LLL+}, and hence
\begin{equation}
\sigma_a^1=\widetilde{\sigma}_a.
\end{equation}

Finally, since only finitely many pairs $(p_j,q_j)$ occur, Proposition~\ref{prop:product-projective-spaces}(ii) shows that, after increasing $a$ once more,
\begin{equation}
\operatorname{dist}\bigl(\mathcal P_a,\mathcal P_a\otimes\mathcal O(p_j,q_j)\bigr)<\frac12
\end{equation}
for every $j=1,\ldots,k$.
Together with \eqref{eq:compatibility-close}, this gives
\begin{equation}
\operatorname{dist}\bigl(\mathcal P_a^t,\mathcal P_a^t\otimes\mathcal O(p_j,q_j)\bigr)<\frac34<1
\end{equation}
for every $t\in[0,1]$ and every $j=1,\ldots,k$.
This proves \eqref{eq:compatibility-bayer}.
\end{proof}

We now prove the main theorem of this section.

\begin{proof}[Proof of Theorem~\ref{thm:polarization-independent}]
Let $H_0,H_1\in\operatorname{Amp}(X)_{\mathbb Q}$.
We take an integer $n>1$ sufficiently large that the classes
\begin{equation}
H_-\coloneqq H_0+\frac{H_0-H_1}{n-1}, \qquad H_+\coloneqq H_1+\frac{H_1-H_0}{n-1}
\end{equation}
are ample; this is possible since the ample cone is open.
By construction, $H_0$ and $H_1$ are positive combinations of $H_-$ and $H_+$.
For $t\in[0,1]$, set
\begin{equation}
\lambda_t\coloneqq (1-t)n+t, \qquad \mu_t\coloneqq (1-t)+tn.
\end{equation}
Then $\lambda_t,\mu_t\geq1$ and
\begin{equation}
\lambda_tH_-+\mu_tH_+=(n+1)\bigl((1-t)H_0+tH_1\bigr).
\end{equation}
Thus the parameter path
\begin{equation}
K\coloneqq \{(\lambda_t,\mu_t)\mid t\in[0,1]\}
\end{equation}
is a compact subset of $\mathbb R_{>0}^2$.

Choose $m>0$ such that $mH_-$ and $mH_+$ are very ample, and let
\begin{equation}
\iota\colon X\hookrightarrow\mathbb P^r\times\mathbb P^s
\end{equation}
be the resulting product embedding.
Recall that $\xi=c_1(\mathcal O(1,0))$ and $\eta=c_1(\mathcal O(0,1))$ are the hyperplane classes from the first and second factors.
Then $\iota^*\xi=mH_-$ and $\iota^*\eta=mH_+$, so
\begin{equation}
\iota^*(\lambda_t\xi+\mu_t\eta)=m(n+1)\bigl((1-t)H_0+tH_1\bigr).
\end{equation}
By Proposition~\ref{prop:product-projective-spaces}, for every $a>0$ the stability conditions
\begin{equation}
\sigma_a^{(\lambda_t,\mu_t)}, \qquad t\in[0,1],
\end{equation}
form a continuous path in $\Stab_{\mathrm{num}}\bigl(\Db(\mathbb P^r\times\mathbb P^s)\bigr)$.

We now show that this path restricts to a path of stability conditions on $\Db(X)$.
Choose a finite resolution of the form
\begin{equation}\label{eq:resolution}
0\to\bigoplus_j\mathcal O(p_{\nu,j},q_{\nu,j})\to\cdots\to\bigoplus_j\mathcal O(p_{1,j},q_{1,j})\to\mathcal O_{\mathbb P^r\times\mathbb P^s}\to\iota_*\mathcal O_X\to0;
\end{equation}
such a resolution exists, for example, by taking a finite bigraded free resolution over the bihomogeneous coordinate ring of $\mathbb P^r\times\mathbb P^s$ and sheafifying.
This resolution gives $\iota$ the cofiltration property of \cite[Definition~3.22]{LLL+} with the parameter $N=\infty$ used there.
We choose $a\gg0$ so that Proposition~\ref{prop:product-projective-spaces}(ii) applies uniformly over $K$ to all the pairs $(p_{\ell,j},q_{\ell,j})$ occurring in \eqref{eq:resolution}, and so that both applications of Lemma~\ref{lem:segre-veronese-compatibility} below hold for these pairs.
As a consequence,
\begin{equation}
\operatorname{dist}\bigl(\mathcal P_a^{(\lambda_t,\mu_t)},\mathcal P_a^{(\lambda_t,\mu_t)}\otimes\mathcal O(p_{\ell,j},q_{\ell,j})\bigr)<1
\end{equation}
for every $t\in[0,1]$ and every $(\ell,j)$.
Hence \cite[Remark~3.18]{LLL+} gives
\begin{equation}
\mathcal P_a^{(\lambda_t,\mu_t)}\prec\mathcal P_a^{(\lambda_t,\mu_t)}\otimes\mathcal O(p_{\ell,j},q_{\ell,j})[1]\preceq\mathcal P_a^{(\lambda_t,\mu_t)}\otimes\mathcal O(p_{\ell,j},q_{\ell,j})[\ell].
\end{equation}
Thus hypothesis~(i) of \cite[Proposition~3.24]{LLL+} is satisfied, while hypothesis~(ii) is automatic since $N=\infty$.
Therefore the stability conditions
\begin{equation}
\iota^\sharp\sigma_a^{(\lambda_t,\mu_t)}, \qquad t\in[0,1],
\end{equation}
form a continuous path of stability conditions on $\Db(X)$ by \cite[Proposition~3.24 and Lemma~3.5(3)]{LLL+}.

We then identify its endpoints.
At $t=0$, the ambient stability condition is $\sigma_a^{(n,1)}$.
Let
\begin{equation}
\jmath_0\colon\mathbb P^r\times\mathbb P^s\hookrightarrow\mathbb P^{N_0}
\end{equation}
be the Segre--Veronese embedding defined by $\mathcal O(n,1)$.
Since
\begin{equation}\label{eq:first-polarization}
c_1\bigl(\iota^*\mathcal O(n,1)\bigr)=m\bigl(nH_-+H_+\bigr)=m(n+1)H_0,
\end{equation}
the composition $\jmath_0\circ\iota$ is a projective embedding associated with a positive multiple of $H_0$.
Apply Lemma~\ref{lem:segre-veronese-compatibility}, with $(\lambda,\mu)=(\lambda_0,\mu_0)=(n,1)$, to the finitely many pairs $(p_{\ell,j},q_{\ell,j})$ from \eqref{eq:resolution}.
The resulting path satisfies the hypotheses of \cite[Proposition~3.24]{LLL+}, and hence induces a continuous path of stability conditions on $\Db(X)$ from $\iota^\sharp\sigma_a^{(n,1)}$ to
\begin{equation}\label{eq:first-distinguished-endpoint}
(\jmath_0\circ\iota)^\sharp\sigma_a^{\mathbb P^{N_0}}.
\end{equation}
Similarly, at $t=1$, the ambient stability condition is $\sigma_a^{(1,n)}$.
Let
\begin{equation}
\jmath_1\colon\mathbb P^r\times\mathbb P^s\hookrightarrow\mathbb P^{N_1}
\end{equation}
be the Segre--Veronese embedding defined by $\mathcal O(1,n)$.
Since
\begin{equation}\label{eq:second-polarization}
c_1\bigl(\iota^*\mathcal O(1,n)\bigr)=m\bigl(H_-+nH_+\bigr)=m(n+1)H_1,
\end{equation}
the composition $\jmath_1\circ\iota$ is a projective embedding associated with a positive multiple of $H_1$.
The same argument, with $(\lambda,\mu)=(\lambda_1,\mu_1)=(1,n)$, gives a continuous path of stability conditions on $\Db(X)$ from $\iota^\sharp\sigma_a^{(1,n)}$ to
\begin{equation}\label{eq:second-distinguished-endpoint}
(\jmath_1\circ\iota)^\sharp\sigma_a^{\mathbb P^{N_1}}.
\end{equation}

It remains to upgrade the restricted paths to the full numerical stability manifold.
The three paths in $\Stab_{\mathrm{num}}\bigl(\Db(\PP^r\times\PP^s)\bigr)$ constructed above concatenate to a path connecting
\begin{equation}
\jmath_0^\sharp\sigma_a^{\mathbb P^{N_0}}, \qquad \sigma_a^{(n,1)}, \qquad \sigma_a^{(1,n)}, \qquad \jmath_1^\sharp\sigma_a^{\mathbb P^{N_1}}.
\end{equation}
The first stability condition admits a mass-Hom bound by \cite[Theorem~7.5]{LLL+}, and hence every stability condition along this path admits a mass-Hom bound by \cite[Remark~7.3]{LLL+}.
Since $\iota$ is finite, \cite[Lemma~7.4(1)]{LLL+} transfers the mass-Hom bound to all the corresponding restricted stability conditions;
and by adjunction, their central charges factor through $\K_{\mathrm{num}}(X)$ as well.
Hence Theorem~\ref{thm:mass-hom-implies-support} shows that all the restricted paths lie in $\Stab_{\mathrm{num}}(\Db(X))$.

By \eqref{eq:first-polarization} and \eqref{eq:second-polarization}, the endpoints \eqref{eq:first-distinguished-endpoint} and \eqref{eq:second-distinguished-endpoint} lie respectively in $\Stab_{\mathrm{num},H_0}^\dagger(\Db(X))$ and $\Stab_{\mathrm{num},H_1}^\dagger(\Db(X))$.
Since these endpoints are connected by the restricted paths constructed above, the two connected components coincide.
\end{proof}

\begin{remark}
The same argument works in the relative setting.
Thus, for $\pi\colon X\to S$ as in \cite[Setup~8.1]{LLL+}, the distinguished component in the full relative numerical stability space is independent of the choice of relatively ample numerical class, provided that $\mathcal N(X/S)$ is a finite-rank free abelian group.
\end{remark}

\bigskip

\bibliographystyle{alphaurl}
\bibliography{ref}

\end{document}